\documentclass[reqno]{amsart}
\usepackage[utf8]{inputenc}
\usepackage{amssymb,amsthm,amsfonts}
\usepackage{amsmath}
\usepackage{tikz}
\usepackage{graphicx}
\usepackage{subcaption}
\usepackage{psfrag}
\usepackage{cancel}
\usepackage{hyperref}
\usepackage{comment}
\usepackage{ulem}
\usepackage{color}
\usepackage{float}

\newtheorem{proposition}{Proposition}[section]

\theoremstyle{definition}
\newtheorem{remark}{Remark}[section]

\theoremstyle{plain}

\numberwithin{equation}{section}

\title[]{Critical points of the multiplier map for the quadratic family}
\author{Anna Belova}
\address{Uppsala University, Uppsala, Sweden}
\email{anna.belova.87@gmail.com}

\author{Igors Gorbovickis}
\address{Jacobs University, Bremen, Germany}
\email{i.gorbovickis@jacobs-university.de}


\thanks{}

\date{\today}

\begin{document}

\begin{comment}
{\it Comments.}
The following are some changes that have been made or need making.
\begin{enumerate}
\item[(1)]
\end{enumerate}

\newpage

\pagestyle{plain}
\setcounter{page}{1}

\end{comment}

\maketitle

\begin{abstract}
The multiplier $\lambda_n$ of a periodic orbit of period $n$ can be viewed as a (multiple-valued) algebraic function on the space of all complex quadratic polynomials $p_c(z)=z^2+c$. 
We provide a numerical algorithm for computing critical points of this function (i.e., points where the derivative of the multiplier with respect to the complex parameter $c$ vanishes). We use this algorithm to compute critical points of $\lambda_n$ up to period~$n=10$.
\end{abstract}

\section{Introduction}

It has been known since the works of Fatou and Julia that multipliers of periodic orbits can carry not only local, but also global information about the holomorphic dynamical system at hand.
In~\cite{Milnor_M2} J.~Milnor used the multipliers of the fixed points to parameterize the moduli space of degree $2$ rational maps. Using this parameterization he proved that this moduli space is isomorphic to $\mathbb C^2$. 
In the attempt to generalize this approach, it was observed by the second author~\cite{Gorbovickis} that the multipliers of any $m-1$ distinct periodic orbits provide a local parameterization of the moduli space of degree $m$ polynomials in a neighborhood of its generic point. It is then a natural question to describe the set of polynomials at which this local parameterization fails, 
that is, to describe the set of all critical points of the multiplier map, defined as the map which assigns to each degree $m$ polynomial the $(m-1)$-tuple of multipliers at the chosen periodic orbits.
The goal of the current paper is to collect numerical data for this problem in the most basic case $m=2$, i.e., the case of the quadratic family 
$$p_c(z)=z^2+c.$$
Even in this case the general problem seems to be quite complicated.

In the current paper we provide a numerical algorithm, that computes critical points of the multiplier map on the space of quadratic polynomials $p_c$. More specifically, given $n\in\mathbb N$, the algorithm finds the values of the parameter $c$, for which the map $p_c$ has a periodic orbit of period $n$, whose multiplier, viewed as a locally analytic function of $c$, has a vanishing derivative. Using this algorithm, we compute critical points of the multiplier map together with the corresponding periodic orbits, for periods up to $n=10$. In particular, we find a complete list of all critical points of the multiplier map, for periods up to $n=8$.

Last but not least, let us mention another important motivation for the current study -- the connection between the critical points of the multiplier map and the hyperbolic components of the famous Mandelbrot set. The argument of quasiconformal surgery implies that appropriate inverse branches of the multiplier map are Riemann mappings\footnote{A Riemann mapping of a simply connected domain is a conformal diffeomorphism of the unit disk onto that domain.} of the hyperbolic components~\cite{Milnor_hyper}. Possible existence of analytic extensions of these Riemann mappings to larger domains might allow to estimate the geometry of the hyperbolic components~\cite{Levin_2009}~\cite{Levin_2011}, which in turn, might shed light on one of the central questions in one-dimensional holomorphic dynamics, the question whether the Mandelbrot set is locally connected. Critical values of the multiplier map are the only obstructions for the above mentioned analytic extensions to exist.

\section{Notation and Terminology}

Let $p_c(z) = z^2 + c$ and denote its $n$-th iteration by $p^{\circ n}_c(z)$.

A point $z$ is {\it a periodic point} of $p_c$, if there exists a positive integer $n$, such that $p^{\circ n}_c(z)=z$. The minimal such $n$ is called \textit{the period} of $z$.

Given $n$, let {\it the period $n$ curve} $\textrm{Per}_n \subset \mathbb{C}\times \mathbb{C}$ be the closure of the locus of points $(c,z)$ such that $z$ is a periodic point of $p_c$ of period $n$ (see ~\cite{MR1755445} for more details).
Observe that each pair $(c,z)\in \mathrm{Per}_n$ determines a periodic orbit
\begin{equation*}
z=z_0\mapsto z_1 \mapsto \cdots \mapsto z_n=z_0.
\end{equation*}

Let $\mathbb{Z}_n$ denote the cyclic group of order $n$. This group acts on $\mathrm{Per}_n$ by cyclicly permuting points of the same periodic orbits for each fixed value of $c$. 
Then the factor space $\mathrm{Per}_n/\mathbb{Z}_n$ consists of pairs $(c, \mathcal{O})$ such that $\mathcal{O}$ is a periodic orbit of $p_c$.
Note that according to~\cite{MR1755445}, the space $\mathrm{Per}_n/\mathbb{Z}_n$ (as well as $\mathrm{Per}_n$) has a structure of a smooth algebraic curve.
(Note that there is a natural projection from $\mathrm{Per}_n$ to $\mathrm{Per}_n/\mathbb{Z}_n$.)

Let $\tilde{\lambda}_n: \mathrm{Per}_n \to \mathbb{C}$ be the map defined by
$$
\tilde\lambda_n\colon (c,z)\mapsto \frac{\partial p^{\circ n}_c}{\partial z}(z) = 2^n z_1 \cdots z_n.
$$
Observe that for all regular points of the projection $(c,z)\mapsto c$, the value $\tilde\lambda_n(c,z)$ is the multiplier of the periodic point $z$. Furthermore, if $z_1$ and $z_2$ belong to the same periodic orbit of $p_c$, then $\tilde\lambda_n(c,z_1)=\tilde\lambda_n(c,z_2)$, hence the map $\tilde\lambda_n$ projects to a well defined map $\lambda_n\colon \mathrm{Per}_n/\mathbb Z_n\to\mathbb C$ that assigns to each pair $(c,\mathcal O)$ the multiplier of the periodic orbit $\mathcal O$.

Both $\lambda_n$ and $\tilde\lambda_n$ are proper algebraic maps (c.f.~\cite{MR1755445}). 
The goal of this work is to study (compute) critical points of the multiplier map $\lambda_n$.

\section{Algorithm for computing critical points of the multiplier map}

\subsection{Computing derivatives}

Observe that all points $(c,z)\in\mathrm{Per}_n$ satisfy the following equation
\begin{equation}
\label{eq:fixed_point}
p^{\circ n}_c (z) = z.
\end{equation}
Together with the Implicit Function Theorem this implies that the parameter $c$ can serve as a local chart on $\mathrm {Per}_n$ at all points $(c,z)\in\mathrm{Per}_n$, such that $\tilde\lambda_n(c,z)\neq 1$. Hence, in a neighborhood of any such point, one can implicitly define a map $z(c)$, so that $(c,z(c))\in\mathrm{Per}_n$ for all nearby values of $c$. Then one can express the multiplier map $\tilde\lambda_n$ in the above local chart as
$$
\tilde\lambda_n(c)=\tilde\lambda_n(c,z(c)).
$$

According to Lemma~4.5 in~\cite{MR1755445}, if $\tilde\lambda_n(c,z)=1$, then $(c,z)$ cannot be a critical point of the multiplier map $\tilde\lambda_n$. Thus, in order to study all critical points of this map, it is sufficient to work in local charts associated with the parameter $c$ (i.e. the critical points of the multiplier map $\tilde\lambda_n$ correspond to those points $(c,z)\in\mathrm{Per}_n$, in a neighborhood of which the map $\tilde\lambda_n(c)=\tilde\lambda_n(c,z(c))$ is defined and $\tilde\lambda_n'(c)=0$).

We will use the following notation for the partial derivatives

\begin{equation*}
\frac{\partial p^{\circ n}_c}{\partial c} := \frac{\partial f_n}{\partial c},
\end{equation*}
where $f_n(c,z) = p^{\circ n}_c(z)$. 
By differentiating both left and right sides of ~\eqref{eq:fixed_point} we get
\begin{equation*}
\frac{dz}{dc}=\frac{\partial p^{\circ n}_c}{\partial c}(z)+\frac{\partial p^{\circ n}_c}{\partial z}(z)\frac{dz}{dc}.
\end{equation*}
Therefore
\begin{equation}
\label{eq:deriv_fix_pnt}
z'=\frac{dz}{dc} = \frac{\partial p^{\circ n}_c}{\partial c}(z)\left(1-\frac{\partial p^{\circ n}_c}{\partial z}(z)\right)^{-1} = \frac{\partial p^{\circ n}_c}{\partial c}(z)\left(1-\lambda_n(c)\right)^{-1}. 
\end{equation}
Observe that $\frac{\partial p^{\circ n}_c}{\partial c}(z)$ satisfies the recurrence relation 
\begin{equation*}
\frac{\partial p^{\circ n}_c}{\partial c}(z) 
=
\frac{\partial}{\partial c}p_c(p^{\circ n-1}_c(z))
=
1+2 \cdot p^{\circ n-1}_c(z)\frac{\partial p^{\circ n-1}_c}{\partial c}(z). 
\end{equation*}
We therefore get an expression for the derivative of the multiplier map
\begin{align}
\label{eq:deriv_mult}
\frac{d\lambda_n}{dc} 
&= \frac{d\lambda_n(c)}{dc} \\
&= 
2^n\left[ z' \cdot p_c(z) \cdot p^{\circ 2}_c(z)\cdots p^{\circ n-1}_c(z) 
\right. \notag\\
&\qquad
\left. + z \cdot \left(\frac{\partial p_c}{\partial c}(z) + z' \frac{\partial p_c}{\partial z}(z) \right)\cdot p^{\circ 2}_c(z)\cdots p^{\circ n-1}_c(z)
\right. \notag\\
&\qquad\qquad\qquad\qquad
\left. + \cdots + z\cdot p_c(z)\cdots p^{\circ n-2}_c(z)\cdot \left(\frac{\partial p^{\circ n-1}_c}{\partial c}(z) + z' \frac{\partial p^{\circ n-1}_c}{\partial z}(z)\right)
\right] \notag\\
&=
2^n\left[z'\prod_{i=1}^{n-1}p^{\circ i}_c(z) + z \sum_{i=1}^{n-1}\left(\frac{dp^{\circ i}_c}{dc}\prod_{\substack{j=1\\ j\neq i}}^{n-1}p^{\circ j}_c(z)\right)\right], \notag
\end{align}
where for $i=1,\dots, n-1$, we denote
\begin{equation*}
\frac{dp^{\circ i}_c}{dc} = \frac{\partial p^{\circ i}_c}{\partial c} + z' \frac{\partial p^{\circ i}_c}{\partial z}.
\end{equation*}
Finally, in order to find the critical points of the multiplier map $\tilde\lambda_n$, we combine~\eqref{eq:fixed_point},~\eqref{eq:deriv_fix_pnt} and~\eqref{eq:deriv_mult} into the following system of three algebraic equations
\begin{equation}
\label{eq:newton}
\left\{\begin{array}{lcl}
p^{\circ n}_c(z)-z
&=&
0\\
z' - \frac{\partial p^{\circ n}_c}{\partial c}(z)\left(1-\frac{\partial p^{\circ n}_c}{\partial z}(z)\right)^{-1} 
&=&
0\\
\frac{d\lambda_n}{dc}
&=&
0,
\end{array}\right.
\end{equation}
with three unknowns $c, z, z'$. Any critical point of the multiplier map $\tilde\lambda_n$ corresponds to a solution of the above system, thus, the problem of finding critical points of the map $\tilde\lambda_n$ can be reduced to the problem of solving the above system.

\subsection{The number of critical points of the multiplier map $\lambda_n$}

A question that naturally arises using the numerical methods, as for example Newton method, is how to make sure that all solutions are found. In this section we derive an upper bound for the number of critical points of the multiplier map which will be later used in the numerical algorithm. 

Let $\nu(n)$ be the number of periodic points of $p_c$ of period $n$ for a generic value of $c$. One can observe that the numbers $\nu(n)$ satisfy the recursive relation
$$
\nu(1)=2\qquad\text{and}\qquad \nu(n)=2^n-\sum_{\substack{m\text{ divides }n,\\ m\neq n}} \nu(m).
$$
In particular, this implies that $\nu(n)\sim 2^n$ as $n\to\infty$.

It was shown in ~\cite{MR1755445} that $c$ can be used as a local uniformizing parameter at any non-parabolic point of $\mathrm{Per}_n/\mathbb{Z}_n$ (i.e., where $\lambda_n(c,\mathcal O) \neq 1$). Moreover, the projection map $\pi_n:\mathrm{Per}_n/\mathbb{Z}_n \to \mathbb{C}$, defined as
$$
\pi_n\colon (c, \mathcal{O})\mapsto c,
$$
is proper of degree $\textrm{deg }\pi_n = \nu(n)/n$. Conversely, in a neighbourhood of a point $(c,\mathcal O)$ with $\lambda_n(c,\mathcal O)=1$, the multiplier $\lambda_n$ serves as a local uniformizing parameter for the curve $\mathrm{Per}_n/\mathbb{Z}_n$ and $\lambda_n: \mathrm{Per}_n/\mathbb{Z}_n \to \mathbb{C}$ is a proper map of degree $\textrm{deg } \lambda_n = \nu(n)/2$.

In order to derive the number of critical points of $\lambda_n$ recall {\it the Riemann-Hurwitz formula} (e.g.~\cite{MR2245223}). Let $X, Y$ be Riemann surfaces, where $X$ is connected with finite-dimensional homology, and $\chi(X), \chi(Y)$ denote the corresponding Euler characteristics. Suppose $f:Y\to X$ is a proper analytic map of degree $\textrm{deg } f$ with finitely many critical values. Then $f$ has finitely many critical points and 
\begin{equation}
\label{r-h}
\chi(Y) = \textrm{deg } f \cdot \chi(X) - \sum_{y\in Y}\left(\textrm{deg}_y f-1\right),
\end{equation}
where $\textrm{deg}_y f \ge 1$ is called {\it a ramification index} (or {\it a local degree}) of $f$ at $y$ and is defined in the following way. There exists an open neighbourhood $U\subset Y$ of $y$, such that $x=f(y)$ has only one preimage in $U$, i.e. $f^{-1}(x) \cap U =\{y\}$, and for all other points $\hat{x}\in f(U)$ the number of preimages is $\textrm{deg}_y f$. Observe that $\textrm{deg}_y f\neq 1$ only at the critical points of $f$ and hence the sum in ~\eqref{r-h} is finite.

Denote the number of (finite) critical points of $\pi_n$ and $\lambda_n$ by $N_{\pi_n}$ and $N_{\lambda_n}$ respectively. Let $Y$ be the Riemann surface obtained from $\mathrm{Per}_n/\mathbb{Z}_n$ by smooth compactification (i.e., compactification in $\mathbb {CP}^n$, possibly followed by resolution of singularities at infinity). Then $Y = \mathrm{Per}_n/\mathbb{Z}_n \cup Z$, where $Z$ is a finite set of points at infinity.


We continuously extend $\pi_n$ to the map $\pi_n: Y \to \mathbb{CP}^1$ of the whole surface $Y$ by setting $\pi_n(z) = \infty$, for all $z\in Z$.

In order to continuously extend the multiplier map in the similar way we need the following Proposition.
\begin{proposition}\label{Mult_at_infty_prop}
	The following relation holds:
	$$\lim_{\substack{(c,\mathcal O) \in \mathrm{Per}_n/\mathbb{Z}_n,\\ c \to \infty}} \lambda_n(c,\mathcal{O})=\infty.$$
\end{proposition}
\begin{proof}
Assume that $c\neq 0$ and denote by $\mathbb{D}_R:=\mathbb{D}_R(0)$ the disc of radius $R = |c|/10$. Then for any $z\in\mathbb{C}\setminus \mathbb{D}_R$ we have
\begin{equation*}
|p_c(z)| = |z^2+c|> R|z|-|c|> (R-10)|z|.
\end{equation*} 
If $|c|$ is sufficiently large, then $R-10>2$, and the above inequality implies that the orbit of any point $z\in\mathbb{C}\setminus \mathbb{D}_R$ converges to $\infty$ under the dynamics of the map $p_c$. In particular, this means that all periodic points of the map $p_c$ lie in the disc $\mathbb{D}_R$.

We now consider the disc $\mathbb{D}_r:=\mathbb{D}_r(0)$ of radius $r=\frac{1}{2}\sqrt{|c|}$. Let $z\in D_r$, i.e. $0 \le |z| \le r$. Observe that
\begin{equation*}
|p_c(z)|=|z^2+c|\ge |c|-|c|/4 > R,
\end{equation*}
i.e. any point $z$ from the disc $\mathbb{D}_r$ is mapped outside of the disc $\mathbb{D}_R$ under one iteration of the map $p_c$ and tends to infinity under further iterations, provided that $|c|$ is sufficiently large. 
Hence all periodic points of the map $p_c$ lie inside the annulus  $\mathbb{D}_R\setminus \mathbb{D}_r$. 

Since $R\to \infty$ and $r\to \infty$ as $c\to \infty$, for all periodic points $z$ of $p_c$ it follows that $|z|\to \infty$ as $c\to \infty$. 

Recall that the multiplier $\lambda_n$ of the periodic point $z$ satisfies 
\begin{equation*}
\lambda_n = \frac{\partial p^{\circ n}_c}{\partial z}(z) = 2^n z_1 \cdots z_n,
\end{equation*}
where $z_1,\dots, z_n$ denotes the points of the orbit of $z$ under the map $p_c$. Combining the above observations, it follows that 
	$$\lim_{\substack{(c,\mathcal O) \in \mathrm{Per}_n/\mathbb{Z}_n,\\ c \to \infty}} \lambda_n(c,\mathcal{O})=\infty.$$
\end{proof}

According to Proposition~\ref{Mult_at_infty_prop}, we continuously extend $\lambda_n$ to the map $\lambda_n: Y \to \mathbb{CP}^1$ of the whole surface $Y$ by setting $\lambda_n(z)=\infty$ for all $z\in {Z}$.


For further reference, let us state the following propositions:
\begin{proposition}\label{deg_pn_prop}
	For any $y\in \mathrm{Per}_n/\mathbb{Z}_n$, we have $\textrm{deg}_y\pi_n \le 2$.
\end{proposition}
\begin{proof}
	If $y=(c,\mathcal O)$ and $\textrm{deg}_y\pi_n > 1$, then by the Implicit Function Theorem we have $\lambda_n(y)=1$, which means that $\mathcal O$ is a parabolic periodic orbit. According to the Fatou-Shishikura inequality (c.f.~\cite{MR2193309}),  the ramification index $\textrm{deg}_y\pi_n$ is not greater than $1$ plus the number of critical points of $p_c$, lying in the basin of attraction of the orbit $\mathcal O$. Since the polynomial $p_c$ has only one critical point, the statement of the proposition follows.
\end{proof}

\begin{proposition}\label{lambda_eq_1_prop}
	If $\lambda_n(c,\mathcal O)=1$, for some point $(c,\mathcal O)\in \mathrm{Per}_n/\mathbb{Z}_n$, then either $(c,\mathcal O)$ is a critical point of $\pi_n$, or $\mathcal O$ is a periodic orbit of period $p<n$, $n=pr$, for some integer $r>1$, and $\lambda_p(c,\mathcal O)$ is a primitive root of unity of degree $r$.
\end{proposition}
\begin{proof}
	The proposition follows from the Fatou-Shishikura inequality in a similar way as Proposition~\ref{deg_pn_prop}.
\end{proof}

We recall that $Z\subset Y$ is the inverse image of infinity under the map $\pi_n$. Applying the Riemann-Hurwitz formula~\eqref{r-h} to the projection $\pi_n$ and using Proposition~\ref{deg_pn_prop}, we get
\begin{align}
\label{r-h_pi}
\chi(Y) &= \textrm{deg }\pi_n \cdot \chi(\mathbb{CP}^1)-N_{\pi_n}-\sum_{y\in Z}\left(\mathrm{deg}_y \pi_n - 1\right).
\end{align}
Therefore, assuming that $Z$ consists of $\kappa$ points, we have
\begin{equation}
\sum_{y\in Z}\left(\textrm{deg}_y \pi_n - 1\right) = \textrm{deg }\pi_n-\kappa,
\end{equation}
and 
\begin{equation}
\label{r-h_pi_2}
\chi(Y) = \textrm{deg }\pi_n-N_{\pi_n}+\kappa.
\end{equation}

Observe that if $(c,\mathcal O)\in\mathrm{Per}_n/\mathbb Z_n$ is a critical point of the projection $\pi_n$, then $\lambda_n(c,\mathcal O)=1$. 
Since all points in the closure of the unit disc $\mathbb{D}$ are regular values of $\lambda_k$ for all $k\in \mathbb{N}$, (c.f.~\cite{Epstein_1}), Proposition~\ref{lambda_eq_1_prop} implies the following formula for the number of critical points of $\pi_n$:

\begin{equation}
\label{N_pi}
N_{\pi_n} = \textrm{deg }\lambda_n-\sum_{\substack{\forall r,p \textrm{ s.t.} \\ n=rp\\p<n}} \textrm{deg }\lambda_p\cdot \varphi(r),
\end{equation}
where $\varphi(r)$ is the Euler's function that counts the positive integers up to $r$ that are relatively prime with $r$.

Analogous computations for $\lambda_n$ using the Riemann-Hurwitz formula~\eqref{r-h} show that
\begin{align}
\label{r-h_lambda}
\chi(Y) &\le \textrm{deg }\lambda_n \cdot \chi(\mathbb{CP}^1)-N_{\lambda_n}-\sum_{y\in Z}\left(\textrm{deg}_y \lambda_n - 1\right)\\
\qquad &=
\textrm{deg }\lambda_n - N_{\lambda_n} + \kappa.
\end{align}
Note that in contrast to the case of the critical points of the projection map $\pi_n$, we are not guaranteed that all critical points of the map $\lambda_n$ are of multiplicity~$1$ (in fact, we do not know whether this is true or false). Because of this we get an inequality instead of an equality in~(\ref{r-h_lambda}).

Now, combining ~\eqref{r-h_pi_2}, ~\eqref{N_pi}, ~\eqref{r-h_lambda}, we derive an upper bound for the number of critical points of the multiplier map
\begin{align*}
N_{\lambda_n} &\le 
\textrm{deg }\lambda_n - \chi(Y) + \kappa \\
\quad &=
\textrm{deg }\lambda_n - \textrm{deg }\pi_n+N_{\pi_n} - \kappa + \kappa \\
\quad &=
2 \cdot \textrm{deg }\lambda_n - \textrm{deg }\pi_n -\sum_{\substack{\forall r,p \textrm{ s.t.} \\ n=rp\\p<n}} \textrm{deg }\lambda_p\cdot \varphi(r).
\end{align*}
Finally, expressing $\textrm{deg }\pi_n$ and $\textrm{deg }\lambda_n$ as $\nu(n)/n$ and $\nu(n)/2$ respectively, we get
\begin{equation}
\label{N_lambda}
N_{\lambda_n} \le \nu(n)-\frac{\nu(n)}{n} -\frac{1}{2}\sum_{\substack{\forall r,p \textrm{ s.t.} \\ n=rp\\p<n}} \nu(p) \cdot \varphi(r).
\end{equation}
\begin{remark}
	We note that the above inequality turns into an equality if the critical points of the multiplier map $\lambda_n$ are counted with their multiplicities.
\end{remark}

\subsection{Algorithm}

To solve the system of equations ~\eqref{eq:newton} we use the Newton method. In addition to the critical points $c$ of the multiplier map $\lambda_n$, this allows us to determine the corresponding periodic points $z$ and its derivatives $z'$.
 
We fix the period $n$. All initial guesses for the Newton method will be randomly chosen on the complex curve defined by the first two equation of~\eqref{eq:newton}. Since every solution to the system~\eqref{eq:newton} lies on this curve, we hope that such initial conditions are more likely to belong to the domains of attraction of the solutions. 

More specifically, initial guesses for the Newton method are chosen as follows: 
we first generate a random guess for the parameter $c$ and compute all periodic points $z$ of period $n$ for the polynomial $p_c$. In order to do this, we apply the algorithm of Hubbard, Schleicher and Sutherland, developed in~\cite{MR1859017}.
For every choice of a periodic point $z$, obtained this way, we compute the corresponding initial value for $z'$ using~\eqref{eq:deriv_fix_pnt}. The resulting triples $(c,z,z')$ are used as initial guesses for the Newton method.

\begin{remark}
Note that since each pair $(c,z)\in \textrm{Per}_n$ determines a pair $(c,\mathcal{O})\in \textrm{Per}_n/\mathbb{Z}_n$, it is enough to consider only one point $z$ (with corresponding $z'$) from each periodic orbit. 
\end{remark}

\begin{remark}
Since system~\eqref{eq:newton} commutes with complex conjugation, each solution $(c,z,z')$ comes together with its complex conjugate $(\overline c,\overline z,\overline z')$. Thus once a solution $(c,z,z')$ is found, we can speed up the computation by also including $(\overline c,\overline z,\overline z')$ to the list of solutions. 
\end{remark}

We can summarize the entire process by the following algorithm. 
\begin{itemize}
\item Input: the period $n$.
	\begin{itemize}
	\item[0:] Set the counter of the critical points of the multiplier map $k=0$, compute the upper bound on $N_{\lambda_n}$ using ~\eqref{N_lambda}.
	\item[1:] Generate randomly $c$, find all $z$ using the method described in ~\cite{MR1859017}, select one $z$ from each orbit, compute $z'$. Store triplets of the initial guesses $c$, $z$, $z'$ in the set $\Sigma_0$.
	\item[2:] If $\Sigma_0\neq\varnothing$, then take an initial guess from the set $\Sigma_0$, remove it from $\Sigma_0$ and proceed to Step 3. If the set $\Sigma_0$ was empty, return to Step 1. 
	\item[3:] Iterate the 3-dimensional Newton operator applied to the system ~\eqref{eq:newton} at the initial guess for maximum 50 times. If the Newton method does not converge after 50 iterations with the desired tolerance, return to Step 2.
	\item[4:] Test if the point $z$ obtained by the Newton method, is of period $n$. If this condition is fulfilled, then store the obtained triplet $(c,z,z')$ and its complex conjugate $(\overline{c},\overline{z},\overline{z}')$ in the set of solutions of ~\eqref{eq:newton} $\Sigma_n$, and set k=k+2.
	\item[5:] If $k < N_{\lambda_n}$, return to Step 2.     
	\end{itemize}
\item Output: the set $\Sigma_n$.
\end{itemize}

\section{Results of the numerical experiments}

The algorithm described above has been implemented in a C++ program. In this section we present the outcome of the numerical experiments. The complete list of critical points of $\lambda_n$ found by the program, can be downloaded from \url{https://www.dropbox.com/sh/nr5847qnhapd8zc/AACdqv2rOxghrBLGQo47zbcma?dl=0}. The tolerance for the Newton's method has been set up at $10^{-10}$.

Note that the multiplier map $\lambda_n$ does not have critical points for periods $n=1,2$. We ran the algorithm for several periods $n=3,\dots,10$. Table~\ref{table1} displays the upper bound for the number of critical points $N_{\lambda_n}$ of the multiplier map for each period and the number of critical points computed by the implemented algorithm, i.e. $\#\Sigma_n$. It is likely that not all critical points have been detected for periods $n=9$ and $n=10$, which can be seen in Table~\ref{table1}. One of the reasons for the missing points might be that the critical points are lying too close to each other and cannot be distinguished using the standard double precision in the computations. However increasing the precision can significantly change the running time. It could also be that the basins of attraction of some of the points for the Newton's method are very small and are easily missed by the initial guesses. The computations for each period up to $n=7$ took less than 30 sec in double precision while for period $n=8$ it took almost 2 hours. Due to the randomness of the initial guesses the running time might be slightly different for the same period for different shots though our experiments showed that this difference is minor.

\begin{table}[h]
\centering
\begin{tabular}{c|l|l|l|l|l|l|l|l}
$n$ & 3 & 4 & 5 & 6 & 7 & 8 & 9 & 10 \\
\hline
Upper bound for $N_{\lambda_n}$ & 2 & 6 & 20 & 38 & 102 & 198 & 436 & 868 \\
\hline
$\#\Sigma_n$ & 2 & 6 & 20 & 38 & 102 & 198 & 434 & 602\\
\hline
Inside the Mandelbrot set (\%) & 0 & 0 & 20 & 10 & 15 & 14 & 14 & 9\\
\hline
Outside the Mandelbrot set (\%) & 100 & 100 & 80 & 90 & 85 & 86 & 86 & 91 
\end{tabular}
\vspace*{4mm}
\caption{Number of critical points of the multiplier map.}\label{table1}
\end{table}

\begin{figure}[h]
	\centering
	\begin{subfigure}[b]{0.45\textwidth}
		\includegraphics[width=0.9\textwidth]{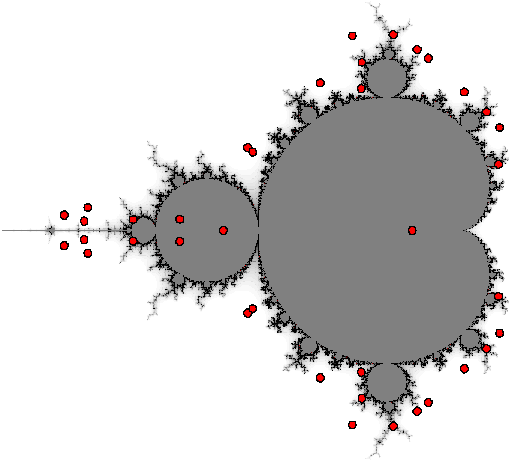}
		\caption{$n=6$}
	\end{subfigure}
	~ 
	\begin{subfigure}[b]{0.45\textwidth}
		\includegraphics[width=0.9\textwidth]{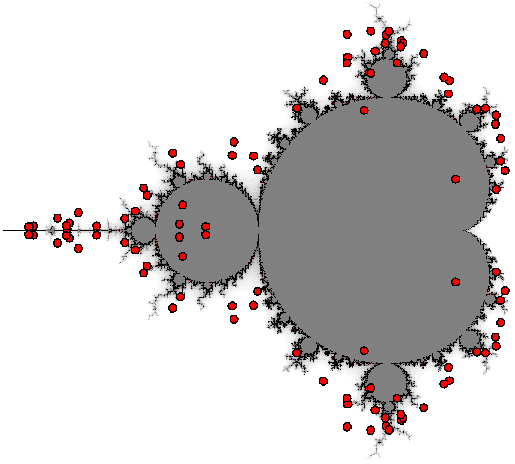}
		\caption{$n=7$}
	\end{subfigure}
	
	~ 
	\begin{subfigure}[b]{0.45\textwidth}
		\includegraphics[width=0.9\textwidth]{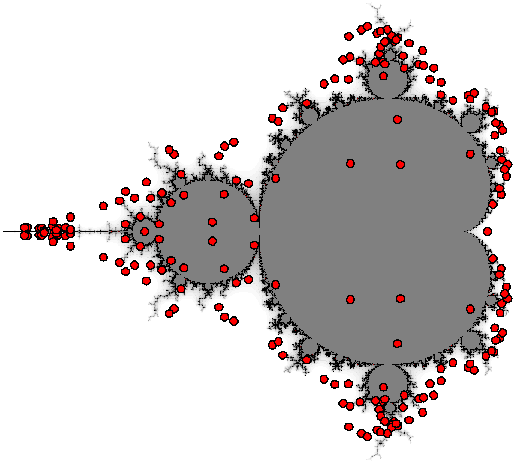}
		\caption{$n=8$}
	\end{subfigure}
	\begin{subfigure}[b]{0.45\textwidth}
		\includegraphics[width=0.9\textwidth]{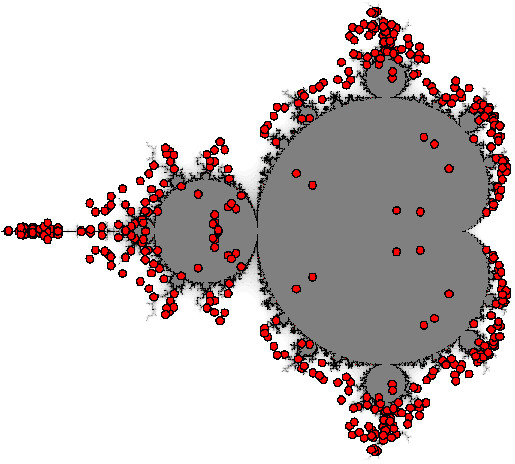}
		\caption{$n=9$}
	\end{subfigure}
	\caption{Critical points of the multiplier map $\lambda_n(c,z)$ on the parameter space.}\label{fig:crit_pnts_mandelbrot}
\end{figure}


\section{Discussion of the results}

In this section we give a basic discussion of the results of our computations and state some questions and conjectures.

Figure~\ref{fig:crit_pnts_mandelbrot} represents critical points of the multiplier map $\lambda_n(c,z)$ on the parameter space of quadratic polynomials for some periods $n$. 
The main question is: do the critical points of the multiplier map have any dynamical meaning? We can see from the pictures that they may correspond to quadratic polynomials with both connected and disconnected Julia sets. The pictures also suggest that as $n$ increases, most of the critical points of the multiplier map tend to accumulate on the boundary of the Mandelbrot set. This leads to the following question:
let $X_n\subset\mathbb C$ be the projection of the set of all critical points of $\lambda_n$ onto the coordinate $c$ and let $\nu_n$ be the probability measure
$$
\nu_n= \frac{1}{\mathrm{card} (X_n)}\sum_{c\in X_n}\delta_c.
$$
Is it true that as $n\to +\infty$, the sequence of measures $\nu_n$ converges to a measure $\mu$ supported on the boundary of the Mandelbrot set? If yes, is $\mu$ the bifurcation measure? Positive answers to such questions have been obtained for various other classes of dynamically significant points for example, in~\cite{Levin} and~\cite{Buff_Gauthier}.

Next, we can observe that while most of the elements of the sets $X_n$ are strictly complex, for periods $n=6, 8$ the sets $X_n$ also contain purely real elements. Can we understand this phenomenon? Furthermore, for $n=6$ one of these purely real critical points lies exactly at $c=0$. The latter suggests the following:
given a periodic point $z_0\neq 0$ of the polynomial $p_0(z)=z^2$, one can compute the derivative of the multiplier $\frac{d\tilde{\lambda}_n}{dc}(0, z_0)$ using the formula
$$
\frac{d\tilde{\lambda}_n}{dc}(0, z_0) = -2^n\sum_{j=0}^{n-1} z_0^{-2^{j+1}},
$$
which was obtained in~\cite{Gorbovickis}. Using this formula, we can check numerically whether $c=0$ is a critical point of the multiplier map $\lambda_n$ for periods $n>8$. Due to the limited precision, we performed computations up to period $n=30$ and obtained that the multiplier map $\lambda_n$ has a critical point at $c=0$, for periods $n=6, 12, 18, 20, 21, 24$ and~$30$. Furthermore, for each of these periods, except $n=6$, the value $d\lambda_n/dc=0$ is obtained at more than one different periodic orbit.

\begin{table}[h]
	\centering
	\begin{tabular}{l|l}
		$n$ & $z_0$ \\
		\hline
		6  & $\exp(2\pi i/9)$ \\    
		12 & $\exp(2\pi i/45)$ \\   
		18 & $\exp(2\pi i/27)$ \\   
		20 & $\exp(2\pi i/25)$ \\   
		21 & $\exp(2\pi i/49)$ \\   
		24 & $\exp(2\pi i/153)$ \\  
		30 & $\exp(2\pi i/99)$ \\
	\end{tabular}
	\vspace*{4mm}
	\caption{Examples of periodic points corresponding to the multiplier map $\lambda_n$ with a critical point at $c=0$.}\label{table3}
\end{table}

\begin{table}[h]
	\centering
	\begin{tabular}{c|l|l|l|l|l|l|l|l}
		$n$ & 3 & 4 & 5 & 6 & 7 & 8 & 9 & 10 \\
		\hline
		$\min |\lambda_n|$ & 7.384 & 5.841 & 4.942 & 4.416 & 4.087 & 3.869 & 3.718 & 3.610 \\
	\end{tabular}
	\vspace*{4mm}
	\caption{The smallest modulus of the critical values of $\lambda_n$ with respect to $n$.}\label{table2}
\end{table}

\begin{figure}[h]
	\centering
	\includegraphics[width=\textwidth]{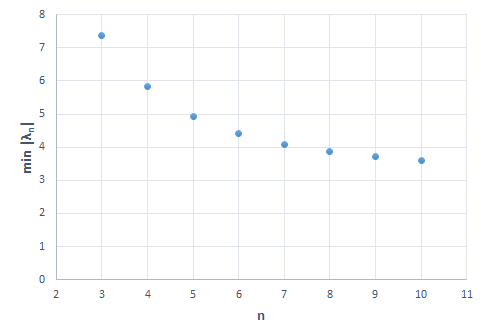}
	\caption{The smallest modulus of the critical values of $\lambda_n$ with respect to $n$.}\label{fig:graph}
\end{figure}

Another important problem is to study the critical values of the multiplier maps $\lambda_n$. As it was mentioned in the introduction, the inverse branches of $\lambda_n$ projected onto the $c$-coordinate are Riemann mappings of the corresponding hyperbolic components of the Mandelbrot set. In particular, this implies that all critical values of the multiplier maps $\lambda_n$ lie outside of the open unit disk. The question is: how close can they get to the unit disk? 
Are the critical values of $\lambda_n$ bounded away from the unit disk uniformly in $n$? If the answer to this question is positive, then one might use the Koebe Distortion Theorem to get uniform bounds on the geometric shape of the hyperbolic components. 
The results of our computations, summarized in Table~\ref{table2} and Figure~\ref{fig:graph}, cannot obviously give a definite answer to the stated question. Nevertheless, Figure~\ref{fig:graph} suggests that the answer might be positive.

\section*{Acknowledgments}

The authors would like to thank the Department of Mathematics at Uppsala University, where the main part of this work has been done. The authors would also like to thank Tanya Firsova for some valuable remarks and suggestions.


\bibliographystyle{plain}
\bibliography{bibl_multipliers}

\end{document}